\documentclass[12pt,a4paper]{amsart}
\usepackage{fullpage,amssymb,mathtools,hyperref}
\makeatletter
\let\@@pmod\pmod
\DeclareRobustCommand{\pmod}{\@ifstar\@pmods\@@pmod}
\def\@pmods#1{\mkern4mu({\operator@font mod}\mkern 6mu#1)}
\makeatother
\newcommand{\C}{\mathbb C}
\newcommand{\Z}{\mathbb Z}
\newcommand{\p}{\mathfrak p}
\newcommand{\q}{\mathfrak q}
\newcommand{\n}{{\mathfrak n}}
\renewcommand{\o}{\mathfrak o}

\renewcommand{\c}{\mathfrak c}
\newcommand{\W}{\mathcal W}
\newcommand{\G}{\mathcal G}
\renewcommand{\AA}{\mathcal A}
\renewcommand{\P}{\operatorname{P}}
\newcommand{\Q}{\operatorname{Q}}
\newcommand{\ZZ}{\operatorname{Z}}
\newcommand{\A}{\operatorname{A}}
\newcommand{\B}{\operatorname{B}}
\newcommand{\T}{\operatorname{T}}
\newcommand{\GL}{\operatorname{GL}}
\newcommand{\U}{\operatorname{U}}
\newcommand{\M}{\operatorname{M}}
\newcommand{\N}{\operatorname{N}}
\renewcommand{\v}{\operatorname{v}}
\DeclareMathOperator{\vol}{vol}
\DeclareMathOperator{\Ind}{Ind}
\DeclareMathOperator{\Hom}{Hom}
\DeclareMathOperator{\diag}{diag}
\theoremstyle{plain}
\newtheorem{theorem}{Theorem}[section]
\newtheorem{corollary}[theorem]{Corollary}
\newtheorem{lemma}[theorem]{Lemma}

\numberwithin{equation}{section}
\theoremstyle{remark}

\begin{document}
\title{Test vectors for Rankin--Selberg $L$-functions}
\author{Andrew R.\ Booker}
\email{andrew.booker@bristol.ac.uk}
\thanks{A.~R.~B.\ was partially supported by EPSRC Grant \texttt{EP/K034383/1}.
M.~L.\ was supported by a Royal Society University Research Fellowship.
No data were created in the course of this study.}
\author{M.\ Krishnamurthy}
\email{muthu-krishnamurthy@uiowa.edu}
\author{Min Lee}
\email{min.lee@bristol.ac.uk}
\address{School of Mathematics\\University of Bristol\\
University Walk\\Bristol\\BS8 1TW\\United Kingdom}
\address{Department of Mathematics\\University of Iowa\\14 MacLean
Hall\\Iowa City, IA 52242-1419\\USA}

\begin{abstract}
We study the local zeta integrals attached to a pair of generic
representations $(\pi,\tau)$ of $\GL_n\times\GL_m$, $n>m$, over a $p$-adic
field. Through a process of unipotent averaging we produce a pair
of corresponding Whittaker functions whose zeta integral is non-zero,
and we express this integral in terms of the Langlands parameters of $\pi$
and $\tau$. In many cases, these Whittaker functions also serve as a
test vector for the associated Rankin--Selberg (local) $L$-function.
\end{abstract}
\maketitle

\section{introduction}\label{intro}
Let $F$ be a non-archimedean local field with ring of integers $\o$
and residue field of cardinality $q$.
For $m<n$, let $\pi$ and $\tau$ be irreducible admissible
representations of $\GL_n(F)$ and $\GL_m(F)$, respectively. We fix an
additive character $\psi$ of $F$ with conductor $\o$ and assume
that $\pi$ and $\tau$ are generic relative to $\psi$.

Recall that the local \emph{zeta integral} $\Psi(s;W,W')$ is defined by
\begin{equation}\label{lz}
\Psi(s;W,W')=\int\limits_{\U_m(F)\backslash\GL_m(F)}
W\!\begin{pmatrix}h&\\&I_{n-m}\end{pmatrix}
W'(h)\|\det{h}\|^{s-\frac{n-m}{2}}\,dh,
\end{equation}
where $W\in\W(\pi,\psi)$ and $W'\in\W(\tau,\psi^{-1})$ are Whittaker
functions in the corresponding Whittaker spaces, and $\U_m$ is the group
of unipotent matrices.
It converges for $\Re(s)\gg1$, and the collection of such zeta integrals spans
a fractional ideal $\C[q^s,q^{-s}]L(s,\pi\boxtimes\tau)$ of the ring
$\C[q^s,q^{-s}]$.
We may choose the generator
to satisfy $1/L(s,\pi\boxtimes\tau)\in\C[q^{-s}]$ and
$\lim_{s\to\infty}L(s,\pi\boxtimes\tau)=1$, and this gives the
local Rankin--Selberg factor attached to the pair $(\pi,\tau)$
in \cite[\S2.7]{JPSS3}.

In particular, if we define a map
\[
\W(\pi,\psi)\otimes\W(\tau,\psi^{-1})\longrightarrow\C(q^{-s})
\]
via
\[
W\otimes W'\mapsto\Psi(s;W,W'),
\]
then there is an element in $\W(\pi,\psi)\otimes\W(\tau,\psi^{-1})$
that maps to $L(s,\pi\boxtimes\tau)$. However, \emph{a priori} this
element need not be a pure tensor.  In  this paper, we produce a
pure tensor $W\otimes W'$ for which the associated zeta integral is
explicitly computable and non-zero.  The precise result that we prove
is the following.
\begin{theorem}\label{main}
Let $\{\alpha_i\}_{i=1}^n$ and $\{\gamma_j\}_{j=1}^m$ denote the
Langlands parameters of $\pi$ and $\tau$, respectively, and let
$L(s,\pi\times\tau)$ be the naive Rankin--Selberg $L$-factor defined by
$$
L(s,\pi\times\tau)=\prod_{i=1}^n\prod_{j=1}^m
\bigl(1-\alpha_i\gamma_jq^{-s}\bigr)^{-1}.
$$
Then there is a pair $(W,W')\in\W(\pi,\psi)\times\W(\tau,\psi^{-1})$,
described explicitly in \S\ref{mc}, such that
\[
\Psi(s;W,W')=L(s,\pi\times\tau).
\]
\end{theorem}

When $\Psi(s;W,W')=L(s,\pi\boxtimes\tau)$, the pair $(W,W')$ is called
a \emph{test vector} for $(\pi,\tau)$. Hence the theorem produces a
test vector whenever $L(s,\pi\times\tau)=L(s,\pi\boxtimes\tau)$---for
instance, if either $\pi$ or $\tau$ is unramified or if
$L(s,\pi\boxtimes\tau)=1$. In general, one has
$L(s,\pi\times\tau)=P(q^{-s})L(s,\pi\boxtimes\tau)$ for a non-zero
polynomial $P\in\C[X]$ (see Lemma~\ref{aux}).

The overview of our method is as follows. Let $\xi^0$
(resp.\ $\varphi^0$) denote the ``essential vector'' in the space of
$\pi$ (resp.\ $\tau$), and let $W_{\xi^0}\in\W(\pi,\psi)$ (resp.\
$W_{\varphi^0}\in\W(\tau,\psi^{-1})$) be the associated essential
Whittaker functions, as described in detail in \S\ref{prelim}.
When $\tau$ is unramified, it follows from \cite[Corollary 3.3]{Mat} that
\begin{equation}\label{unr3}
L(s,\pi\boxtimes\tau)
=\int\limits_{\U_m(F)\backslash\GL_m(F)}
W_{\xi^0}\!\begin{pmatrix}h&\\&I_{n-m}\end{pmatrix}
W_{\varphi^0}(h)\|\det{h}\|^{s-\frac{n-m}{2}}\,dh
\end{equation}
for a suitable normalization of the measure on
$\U_m(F)\backslash\GL_m(F)$.  When $m=n-1$, the above equality is part
of the characterization of the essential vector in \cite{JPSS,J1}; the
fact that it holds for any $m<n$ is the result of a concrete realization
of essential functions in \cite{Mat}.
On the other hand, if $\tau$ is ramified then the local integral
in \eqref{unr3} vanishes.  Through a process of unipotent averaging
(see \eqref{uav} below), we modify $W_{\xi^0}$ to obtain a Whittaker
function $W\in\W(\pi,\psi)$ such that the resulting zeta integral
$\Psi(s;W,W_{\varphi^0})$ equals $cL(s,\pi\times\tau)$ for a non-zero
number $c\in\C$, depending on the conductor of $\tau$, its central
character $\omega_\tau$, and $\psi$. Setting $W'=c^{-1}W_{\varphi^0}$,
we obtain the required pair $(W,W')$.

We mention some related results in the literature. First,
if $\pi$ and $\tau$ are discrete series representations, then the
existence of a test vector $(W,W')$ was shown in \cite{Mat2}, but
the Whittaker function $W'$ there is taken to be in a larger space,
namely the Whittaker space associated to the standard module of $\tau$.
Second, the so-called \emph{local Birch lemma}, arising in the context
of $\p$-adic interpolation of special values of twisted Rankin--Selberg
(global) $L$-functions, is also related.  It concerns evaluation of a
local integral in the special case that $\pi$ is unramified and $\tau$ is
the twist of an unramified representation by a character with non-trivial
conductor; see \cite[Proposition 3.1]{KMS} and \cite[Theorem 2.1]{Jan}.
The approach in \cite{KMS} is similar to ours in that it also uses
a process of unipotent averaging in order to modify the Whittaker
function on the larger general linear group. We can of course apply
Theorem~\ref{main} to their setup: Since $L(s,\pi\times\tau)=1$ in this
case, the pair $(W,W_{\varphi^0})$ described above has the property that
$\Psi(s;W,W_{\varphi^0})$ is an explicit constant (independent of $s$).

Finally, note that one can obtain a global version of Theorem \ref{main} by
combining the test vectors at all (finite) places.  In work in progress,
we study the analogous question over an archimedean local field.

\subsection*{Acknowledgements}
The second author (M.~K.) would like to thank A.~Raghuram for pointing
him to \cite{KMS}.

\section{preliminaries}\label{prelim}
Let $\p$ be the unique maximal ideal in $\o$. We fix a generator
$\varpi$ of $\p$ with absolute value $\|\varpi\|=q^{-1}$. Let
$\v:F^\times\rightarrow\Z$ denote the valuation map, and extend it to
fractional ideals in the usual way. For any
$n>1$, let $\B_n=\T_n\U_n$ be the Borel subgroup of $\GL_n$
consisting of upper triangular matrices; let $\P_n'\supset\B_n$ be the
standard parabolic subgroup of type $(n-1,1)$ with Levi decomposition
$\P_n'=\M_n\N_n$. Then $\M_n\cong\GL_{n-1}\times\GL_1$
and
\[
\N_n=\left\{\begin{pmatrix}I_{n-1}&\ast\\&1\end{pmatrix}\right\}.
\]
Also, we write $\ZZ_n$ to denote the center consisting of scalar matrices
and $\A_n\subset\T_n$ to denote the subtorus consisting of diagonal
matrices with lower-right entry $1$.

If $R$ is any $F$-algebra and $H$ is any algebraic $F$-group, we
write $H(R)$ to denote the corresponding group of $R$-points. Let
$\P_n(R)\subset\P_n'(R)$ denote the mirabolic subgroup consisting of
matrices whose last row is of the form $(0,\ldots,0,1)$, i.e.,
$$
\P_n(R)=\left\{\begin{pmatrix}h&y\\&1\end{pmatrix}:
h\in\GL_{n-1}(R), y\in R^{n-1}\right\}
\cong\GL_{n-1}(R)\ltimes\N_n(R).
$$
The character
\[
u\mapsto\psi\!\left(\sum_{i=1}^{n-1}u_{i,i+1}\right)\quad\text{for }u\in\U_n(F)
\]
defines a \emph{generic character} of $\U_n(F)$, and by abuse of
notation we continue to denote this character by $\psi$. Further,
for any algebraic subgroup $V\subseteq\U_n$, $\psi$ defines a
character of $V(F)$ via restriction. In particular, we may consider
the character $\psi|_{\N_n(F)}$; its stabilizer in $\M_n(F)$ is then
$\P_{n-1}(F)$, where we regard $\P_{n-1}$ as a subgroup of $\M_n$ via
$h\mapsto\begin{psmallmatrix}h&\\&1\end{psmallmatrix}$.

An irreducible representation $(\pi,V_\pi)$ of $\GL_n(F)$ is said to
be \emph{generic} if
\[
\Hom_{\GL_n(F)}\bigl(V_\pi,\Ind_{\U_n(F)}^{\GL_n(F)}\psi\bigr)\ne0.
\]
By Frobenius reciprocity, this means that there is a non-zero linear form
$\lambda:V_\pi\to\C$ satisfying $\lambda(\pi(u)v)=\psi(u)\lambda(v)$
for $v\in V_\pi$, $u\in\U_n(F)$.  It is known (see \cite{G-K}) that for
a generic $\pi$ the space of such linear functionals, or equivalently
the space $\Hom_{\GL_n(F)}(V_\pi,\Ind_{\U_n(F)}^{\GL_n(F)}\psi)$,
is of dimension $1$. Let $\W(\pi,\psi)$ denote the Whittaker model
of $\pi$, viz.\ the space of functions $W_v$ on $\GL_n(F)$ defined
by $W_v(g)=\lambda(\pi(g)v)$ for $v\in V_\pi$. Then $\W(\pi,\psi)$
is independent of the choice of $\lambda$, and for $u\in\U_n(F)$,
$g\in \GL_n(F)$,
\begin{align*}
W_v(ug)&=\psi(u)W(g),\\
W_v(g)&=W_{\pi(g)v}(I_n).
\end{align*}

We will consider certain compact open subgroups of $\GL_n(F)$;
namely, for any integer $f\ge 0$, set
\begin{align*}
K_1(\p^f)&=\left\{g\in\GL_n(\o):g\equiv
\begin{psmallmatrix}&&&\ast\\&\ast&&\vdots\\&&&\ast\\0&\cdots&0&1\end{psmallmatrix}
\pmod*{\p^f}\right\},\\
K_0(\p^f)&=\left\{g\in\GL_n(\o):g\equiv
\begin{psmallmatrix}&&&\ast\\&\ast&&\vdots\\&&&\ast\\0&\cdots&0&\ast\end{psmallmatrix}
\pmod*{\p^f}\right\},
\end{align*}
so that $K_1(\p^f)$ is a normal subgroup of $K_0(\p^f)$,
with quotient $K_0(\p^f)/K_1(\p^f)\cong(\o/\p^f)^\times$.

Next we introduce our choice of measures. For $n\ge1$ we
normalize the Haar measure on $\GL_n(F)$ and $\GL_n(\o)$ so that
$\vol(\GL_n(\o))=1$, and we fix the Haar measure on $\U_n(F)$ for which
$\vol(\U_n(F)\cap\GL_n(\o))=1$. From these, we obtain a right-invariant
measure on $\U_n(F)\backslash\GL_n(F)$. We may make this explicit
using the Iwasawa decomposition. For instance, let $dx$ be the Haar
measure on $F$ such that $\o$ has unit volume, and let $d^\times x$ be
the multiplicative measure on $F^\times$ such that $\vol(\o^\times)=1$,
i.e., $d^\times x=\frac{q}{q-1}\frac{dx}{\|x\|}$. Let $dz$ and $da$ be the
corresponding measures on the center $\ZZ_n(F)\cong F^\times$
and the subtorus $\A_n(F)\cong(F^\times)^{n-1}$, respectively.  We fix the
isomorphism $(F^\times)^{n-1}\cong\A_n(F)$ via $(a_1,\ldots,a_{n-1})\mapsto
a=t(a_1,\ldots,a_{n-1})$, where
\begin{equation}\label{a}
t(a_1,\ldots,a_{n-1})=\left(\begin{matrix}a_1a_2\ldots a_{n-1}&&&&\\
&a_1a_2\ldots a_{n-2}&&&\\&&\ddots&&\\&&&a_1&\\&&&&1\end{matrix}\right).
\end{equation}
Then $da=d^\times a_1\,d^\times a_2\cdots d^\times a_{n-1}$.
If $f\in C_c^\infty(\GL_n(F))$ is $\U_n(F)$-invariant on the left,
we then have the integration formula
\begin{equation}\label{Haar}
\int\limits_{\U_n(F)\backslash\GL_n(F)}f(g)\,dg
=\int\limits_{\ZZ_n(F)\times\A_n(F)\times\GL_n(\o)}f(zak)\delta_{\B_n}(a)^{-1}
\,dz\,da\,dk,
\end{equation}
where $\delta_{\B_n}$ is the modulus character, defined so that
\begin{equation}\label{delta}
\delta_{\B_n}(a)=\prod_{i=1}^{n-1}\|a_i\|^{i(n-i)}.
\end{equation}

Next we review the notion of \emph{conductor} and
the theory of the \emph{essential vector} associated to an irreducible,
admissible, generic representation $\pi$. According
to \cite{JPSS} (see also \cite{J1}), there is a unique positive integer
$m(\pi)$ such that the space of $K_1\bigl(\p^{m(\pi)}\bigr)$-fixed
vectors is $1$-dimensional. Further, as alluded to in the introduction,
by loc.~cit.\ there is a unique vector $\xi^0$ in this space, called the
essential vector, with the associated essential function
$W_{\xi^0}\in\W(\pi,\psi)$ satisfying the condition
$W_{\xi^0}\begin{psmallmatrix}gh&\\&1\end{psmallmatrix}
=W_{\xi^0}\begin{psmallmatrix}g&\\&1\end{psmallmatrix}$
for all $h\in\GL_{n-1}(\o)$ and $g\in\GL_{n-1}(F)$.
Since $\U_n$ acts via $\psi$ on the left, it follows that
\begin{equation}\label{support}
W_{\xi^0}\bigl(t(a_1,\ldots,a_{n-1})\bigr)\ne0
\implies
a_1,\ldots,a_{n-1}\in\o.
\end{equation}

If $\pi$ is unramified, let
$W_\pi^{0,\psi}\in\W(\pi,\psi)$ denote the normalized spherical function
\cite[p.~2]{J1}. If $m(\pi)=0$ then by uniqueness of essential functions,
one has the equality $W_{\xi^0}=W_\pi^{0,\psi}$. The integral ideal
$\p^{m(\pi)}$ is called the \emph{conductor} of $\pi$.  In passing, we
mention that the integer $m(\pi)$ can also be characterized as the degree
of the monomial in the local $\epsilon$-factor $\epsilon(s,\pi,\psi)$
\cite{JPSS}, i.e.\ so that
\[
\epsilon(s,\pi,\psi)=\epsilon(\pi,\psi)q^{m(\pi)(\frac12-s)}
\]
for some $\epsilon(\pi,\psi)\in\C^\times$.

A crucial property of the conductor is that $K_0(\p^m(\pi))$ acts on
the space of $K_1(\p^{m(\pi)})$-fixed vectors via the central character
$\omega_\pi$ (cf.~\cite[Section 8]{CPS1}).
Precisely, for
$g=(g_{i,j})\in K_0\bigl(\p^{m(\pi)}\bigr)$, define
\[
\chi_\pi(g)=\begin{cases}
1&\text{if }m(\pi)=0,\\\omega_\pi(g_{n,n})&\text{if }m(\pi)>0.
\end{cases}
\]
It is shown in loc.~cit.\ that $\chi_\pi$ is a character of
$K_0(\p^{m(\pi)})$ trivial on $K_1(\p^{m(\pi)})$, and
\[
\pi(g)\xi^0=\chi_\pi(g)\xi^0
\quad\text{for all }g\in K_0\bigl(\p^{m(\pi)}\bigr).
\]

We end this section by recalling the definition of conductor of a
multiplicative character $\chi$ of $F^\times$. If $\chi$ is trivial on
$\o^\times$ then the conductor of $\chi$ is $\o$; otherwise, the conductor
is $\p^n$, where $n\ge1$ is the least integer such that $\chi$
is trivial on $1+\p^n$.

\subsection{Rankin--Selberg $L$-functions}\label{RS}
In this subsection alone we drop the assumption that $m<n$ and allow
$(m,n)$ to be an arbitrary pair of positive integers.  For $\pi$
and $\tau$ irreducible, admissible, generic representations of
$\GL_n(F)$ and $\GL_m(F)$, respectively, let $L(s,\pi\boxtimes\tau)$
be as defined in \cite{JPSS3}. When $m<n$, $L(s,\pi\boxtimes\tau)$
is defined as in the introduction. For $m>n$, one defines
$L(s,\pi\boxtimes\tau)=L(s,\tau\boxtimes\pi)$. For $m=n$, the defining
local integrals are different and involve a Schwartz function on $F^n$;
see loc.~cit.

Next we elaborate on the definition of the naive Rankin--Selberg
$L$-factor, $L(s,\pi\times\tau)$, introduced in Theorem~\ref{main}.
By definition, the $L$-function $L(s,\pi)$
is of the form $P_\pi(q^{-s})^{-1}$, where $P_\pi\in\C[X]$ has
degree at most $n$ and satisfies $P_\pi(0)=1$. We may then find $n$
complex numbers $\{\alpha_i\}_{i=1}^n$ (allowing some of them to be zero)
satisfying
\[
L(s,\pi)=\prod_{i=1}^n(1-\alpha_{i}q^{-s})^{-1}.
\]
We call the set $\{\alpha_i\}$ the \emph{Langlands parameters}
of $\pi$; if $\pi$ is spherical, they agree with the usual Satake
parameters. Let $\{\gamma_j\}_{j=1}^m$ be the Langlands parameters of
$\tau$, and set
\begin{equation}\label{e:naive_L}
L(s,\pi\times\tau)=\prod_{i=1}^n\prod_{j=1}^m(1-\alpha_i\gamma_jq^{-s})^{-1}.
\end{equation}
Of course, $L(s,\pi\times\tau)=L(s,\pi\boxtimes\tau)$
if both $\pi$ and $\tau$ are spherical.

In the following lemma we describe the connection between
$L(s,\pi\times\tau)$ and $L(s,\pi\boxtimes\tau)$.  To that end, we first
recall the classification of irreducible admissible representations
of $\GL_n(F)$.  Let $\AA_n$ denote the set of equivalence classes of
such representations, and put $\AA=\bigcup\AA_n$.  The essentially
square-integrable representations of $\GL_n(F)$ have been classified
by Bernstein and Zelevinsky, and they are as follows.  If $\sigma$
is an essentially square-integrable representation of $\GL_n(F)$, then
there is a divisor $a\mid n$ and a supercuspidal representation $\eta$
of $\GL_a(F)$ such that if $b=\frac{n}{a}$ and $\Q$ is the standard
(upper) parabolic subgroup of $\GL_n(F)$ of type $(a,\ldots,a)$, then
$\sigma$ can be realized as the unique quotient of the (normalized)
induced representation
\[
\Ind_{\Q}^{\GL_n(F)}(\eta,\eta\|\cdot\|,\ldots,\eta\|\cdot\|^{b-1}).
\]
The integer $a$ and the class of $\eta$ are uniquely determined
by $\sigma$.  In short, $\sigma$ is parametrized by $b$ and $\eta$,
and we denote this by $\sigma=\sigma_b(\eta)$; further, $\sigma$ is
square-integrable (also called ``discrete series'') if and only if
the representation $\eta\|\cdot\|^{\frac{b-1}{2}}$ of $\GL_a(F)$
is unitary.

Now, let $\P$ be an upper parabolic subgroup of $\GL_n(F)$
of type $(n_1,\ldots,n_r)$.
For each $i=1,\ldots,r$, let $\tau_i^0$ be a discrete series representation
of $\GL_{n_i}(F)$.
Let $(s_1,\ldots,s_r)$ be a sequence of
real numbers satisfying $s_1\ge\cdots\ge s_r$, and put
$\tau_i=\tau_i^0\otimes\|\cdot\|^{s_i}$
(an essentially square-integrable representation).
Then the induced representation
\[
\xi=\Ind_{\P}^{\GL_n(F)}(\tau_1\otimes\cdots\otimes\tau_r)
\]
is said to be a representation of $\GL_n(F)$ of Langlands type.
If $\tau\in\AA_n$, then it is well known that it is uniquely representable
as the quotient of an induced representation of Langlands type.
We write $\tau=\tau_1\boxplus\cdots\boxplus\tau_r$ to
denote this realization of $\tau$.
Thus one obtains a sum operation on the set $\AA$ \cite[\S9.5]{JPSS3}.
It follows easily from the definition
that $L(s,\pi\times\tau)$ is bi-additive, i.e.\
\begin{align*}
L(s,\pi\times(\tau\boxplus\tau'))
&=L(s,\pi\times\tau)L(s,\pi\times\tau')\\
L(s,(\pi\boxplus\pi')\times\tau)
&=L(s,\pi\times\tau)L(s,\pi'\times\tau)
\end{align*}
for all $\pi,\pi',\tau,\tau'\in\AA$.  The local factor
$L(s,\pi\boxtimes\tau)$ is also bi-additive in the above sense, by
\cite[\S9.5, Theorem]{JPSS3}.

\begin{lemma}\label{aux}
Let $m$ and $n$ be positive integers, and
consider $\pi\in\AA_n$, $\tau\in\AA_m$.
Then
\[
L(s,\pi\times\tau)=P(q^{-s})L(s,\pi\boxtimes\tau)
\]
for a polynomial $P\in\C[X]$ (depending on $\pi$ and $\tau$)
satisfying $P(0)=1$.
\end{lemma}
\begin{proof}
Since $\pi$ and $\tau$ are sums of essentially square-integrable
representations and $L(s,\pi\times\tau)$ and $L(s,\pi\boxtimes\tau)$
are both additive with respect to $\boxplus$, it suffices to prove
the lemma for a pair $(\pi,\tau)$ of essentially square-integrable
representations. In particular, assume $\pi=\sigma_b(\eta)$ as above.

We proceed by induction on $m$. If
$m=1$ then $\tau=\chi$ is a quasi-character of $F^\times$
and $L(s,\pi\boxtimes\tau)=L(s,\pi\otimes\chi)$, where
$\pi\otimes\chi$ is the representation of $\GL_n(F)$ defined by
$g\mapsto\pi(g)\chi(\det g)$. If $\chi$ is unramified, then
\[
L(s,\pi\otimes\chi)=L(s,\pi\times\chi),
\]
and consequently $P=1$. On the other hand, if $\chi$ is
ramified then $L(s,\pi\times\chi)=1$, and the assertion follows
since $L(s,\pi\otimes\chi)^{-1}$ is a polynomial in $q^{-s}$.

We now assume $m>1$ and $\tau$ is an essentially square-integrable
representation of $\GL_m(F)$, say $\tau=\sigma_{b'}(\eta')$, where
$\eta'\in\AA_{a'}$ is supercuspidal and $a'b'=m$. Then the standard
$L$-factor $L(s,\tau)$ is given by $L(s,\tau)=L(s+b'-1,\eta')$
\cite{JPSS3}. Therefore, $L(s,\tau)=1$ unless $a'=1$ and $\eta'=\chi$
is an unramified quasi-character of $F^\times$. On the other hand,
if $L(s,\tau)=1$, then $L(s,\pi\times\tau)=1$ and the assertion of the
lemma follows. Hence we may assume $\tau=\sigma_m(\chi)$ for an unramified
quasi-character $\chi$ of $F^\times$, in which case
\begin{equation}\label{e1}
\begin{aligned}
L(s,\pi\times\tau)&=L(s,\pi\otimes\chi\|\cdot\|^{m-1})
=L(s,\sigma_b(\eta)\otimes\chi\|\cdot\|^{m-1})\\
&=L(s+m-1+b-1,\eta\otimes\chi).
\end{aligned}
\end{equation}
On the other hand, it follows from \cite[\S8.2, Theorem]{JPSS3} that
\begin{equation}\label{e2}
L(s,\pi\boxtimes\tau)=\begin{cases}
\prod_{j=0}^{m-1}L(s+j+b-1,\eta\otimes\chi)&\text{if }m\leq n,\\
\prod_{i=0}^{b-1}L(s+m-1+i,\eta\otimes\chi)&\text{if }m>n.
\end{cases}
\end{equation}
From \eqref{e1} and \eqref{e2}, one sees that the ratio
$\frac{L(s,\pi\times\tau)}{L(s,\pi\boxtimes\tau)}$ is a polynomial
in $q^{-s}$, thus proving the lemma.
\end{proof}

\begin{corollary}
If $L(s,\pi\boxtimes\tau)=1$ then either $L(s,\pi)=1$ or $L(s,\tau)=1$.
\end{corollary}
\begin{proof}
If $L(s,\pi\boxtimes\tau)=1$ then Lemma~\ref{aux} implies that
$L(s,\pi\times\tau)$ is a polynomial in $q^{-s}$, and hence must
be $1$. This in turn implies the conclusion.
\end{proof}

\section{the main calculation}\label{mc}
Recall that $\xi^0$ and $\varphi^0$ are the essential vectors
of $\pi$ and $\tau$, respectively.  Here we construct a pair
$(W,W')\in\W(\pi,\psi)\times\W(\tau,\psi^{-1})$ as in Theorem~\ref{main}.
Let $\n$, $\q$ and $\c$ denote the conductors of $\pi$, $\tau$ and
$\omega_\tau$, respectively.  If $\tau$ is an unramified representation
of $\GL_m(F)$, then by \eqref{unr3} we have
\[
\Psi(s; W_{\xi^0},W_{\varphi^0})=L(s,\pi\boxtimes\tau)
=L(s,\pi\times\tau).
\]
Thus, in this case we can take $(W,W')=(W_{\xi^0},W_{\varphi^0})$.

Let us assume from now on that $\tau$ is ramified, meaning $\v(\q)>0$.
Since $\c\supseteq\q$, we have $\v(\c)\le\v(\q)$.
Consider $\beta=(\beta_1,\ldots,\beta_m)\in F^m$, with
$\beta_i\in\q^{-1}$ for $i=1,\ldots,m$, and let $u(\beta)$ denote the
$n\times n$ matrix with $1$s on the diagonal and $\beta^t$ embedded above
the diagonal in the $(m+1)^{\text{st}}$ column.
Let $\xi^0_\beta$ denote the vector $\xi_\beta^0=\pi(u(\beta))\xi^0$,
and define
\begin{equation}\label{uav}
\overline{\xi}=\frac1{[\o:\q]^{m-1}}
\sum_{(\beta_1,\ldots,\beta_{m-1})\in(\q^{-1}/\o)^{m-1}}
\xi^0_{(\beta_1,\ldots,\beta_{m-1},\varpi^{-\v(\c)})}.
\end{equation}
(When $m=1$ we understand there to be one summand, so that
$\overline{\xi}=\xi^0_{(\varpi^{-\v(\c)})}$.)
We will now calculate $\Psi(s; W_{\overline{\xi}},W_{\varphi^0})$,
which by linearity equals
\[
\frac1{[\o:\q]^{m-1}}
\sum_{(\beta_1,\ldots,\beta_{m-1})\in(\q^{-1}/\o)^{m-1}}
\Psi\bigl(s;W_{\xi^0_{(\beta_1,\ldots,\beta_{m-1},\varpi^{-\v(\c)})}},W_{\varphi^0}\bigr).
\]

Put $K=\GL_m(\o)$.
By \eqref{Haar}, for fixed $\beta=(\beta_1,\ldots,\beta_m)$, we have
\begin{align*}
\Psi(s;W_{\xi^0_\beta},W_{\varphi^0})
&=\int\limits_{\ZZ_m(F)\times\A_m(F)\times K}\omega_\tau(z)
W_{\xi^0}\!\left(\begin{pmatrix}zak&\\&I_{n-m}\end{pmatrix}u(\beta)\right)
W_{\varphi^0}(ak)\\
&\qquad\qquad\qquad\cdot\delta_{\B_m}(a)^{-1}\|\det{(za)}\|^{s-\frac{n-m}{2}}
\,dz\,da\,dk\\
&=\int\limits_{F^{\times}\times\A_m(F)\times K}
\omega_\tau(z)\left(\sum_{j=1}^{m}\psi(zk_{j}\beta_j)\right)
W_{\xi^0}\!\begin{pmatrix}za&\\&I_{n-m}\end{pmatrix}
W_{\varphi^0}(ak)\\
&\qquad\qquad\qquad\cdot
\delta_{\B_m}(a)^{-1}\|z^m\det{a}\|^{s-\frac{n-m}{2}}\,d^\times z\,da\,dk,
\end{align*}
where $(k_1,\ldots,k_m)$ is the bottom row of the matrix $k$.
Here we have used the fact that the function $h\mapsto
W_{\xi^0}\begin{psmallmatrix}h&\\&I_{n-m}\end{psmallmatrix}$,
$h\in\GL_m(F)$, is right $K$-invariant.  Now, performing the average over
$\beta_j\in\q^{-1}/\o$ for each $j<m$,
we see that $\Psi(s;W_{\overline{\xi}},W_{\varphi^0})$ equals
\begin{align*}
\int\limits_{\left\{\substack{F^\times\times\A_m(F)\times K\\
zk_j\in\q\;\forall j<m}\right\}}
&\omega_\tau(z)\psi\bigl(zk_{m}\varpi^{-\v(\c)}\bigr)
W_{\xi^0}\!\begin{pmatrix}za&\\&I_{n-m}\end{pmatrix}
W_{\varphi^0}(ak)\\
&\cdot\delta_{\B_m}(a)^{-1}\|z^m\det{a}\|^{s-\frac{n-m}{2}}\,d^\times z\,da\,dk.
\end{align*}
By \eqref{support} we have $W_{\xi^0}(t(a_1,\ldots,a_{n-1}))=0$ unless
$a_1,\ldots,a_{n-1}\in\o$. In view of \eqref{a}, it follows that
the integrand vanishes unless $z$ is integral.

Note that
\[
zk_j\in\q\;\forall j<m\iff
k\in K_0(z^{-1}\q\cap\o).
\]
For $r\in\Z_{\ge0}$, put
\begin{align*}
\Psi_r=\int\limits_{\A_m(F)\times K_0(\q\p^{-r}\cap\o)}
&\G(\omega_\tau,\psi,\varpi^{r-\v(\c)}k_m)
W_{\xi^0}\!\begin{pmatrix}\varpi^ra&\\&I_{n-m}\end{pmatrix}
W_{\varphi^0}(ak)\\
&\cdot
\delta_{\B_m}(a)^{-1}\|\det{a}\|^{s-\frac{n-m}{2}}
\,da\,dk,
\end{align*}
where $\G(\omega_\tau,\psi,y)$ denotes the Gauss sum
\[
\G(\omega_\tau,\psi,y)=\int_{\o^\times}\omega_\tau(z)\psi(yz)\,d^\times z.
\]
Then we have
\[
\Psi(s;W_{\overline{\xi}},W_{\varphi^0})=\sum_{r\geq 0}
\omega_\tau(\varpi^r)q^{-rm(s-\frac{n-m}{2})}\Psi_r.
\]

Suppose that $\omega_\tau$ is ramified, so that
$\v(\c)>0$. Then
$\G(\omega_\tau,\psi,\varpi^{r-\v(\c)}k_m)$ vanishes unless
$\v(\varpi^{r-\v(\c)}k_m)=-\v(\c)$, which implies $\v(k_m)=-r$.
Since $k_m$ is integral, it
follows that $r=0$ is the only contributing term to
$\Psi(s;W_{\overline{\xi}},W_{\varphi^0})$, so that
\begin{align*}
\Psi(s;W_{\overline{\xi}},W_{\varphi^0})
=\int\limits_{\A_m(F)\times K_0(\q)}
&\G\bigl(\omega_\tau,\psi,\varpi^{-\v(\c)}k_m\bigr)
W_{\xi^0}\!\begin{pmatrix}a&\\&I_{n-m}\end{pmatrix}
W_{\varphi^0}(ak)\\
&\cdot\delta_{\B_m}(a)^{-1}\|\det{a}\|^{s-\frac{n-m}{2}}
\,da\,dk.
\end{align*}
Moreover, since $k_m\in\o^\times$, we have
$\G(\omega_\tau,\psi,\varpi^{-\v(\c)}k_m)=\omega_\tau(k_m)^{-1}
\G(\omega_\tau,\psi,\varpi^{-\v(\c)})$, and thus
\begin{equation}\label{Psi}
\Psi(s;W_{\overline{\xi}},W_{\varphi^0})
=c\int_{\A_m(F)}W_{\xi^0}\!\begin{pmatrix}a&\\&I_{n-m}\end{pmatrix}
W_{\varphi^0}(a)\delta_{\B_m}(a)^{-1}\|\det{a}\|^{s-\frac{n-m}{2}}\,da,
\end{equation}
where
\[
c=\frac{\G\bigl(\omega_\tau,\psi,\varpi^{-\v(\c)}\bigr)}{[\GL_m(\o):K_0(\q)]}
\ne0.
\]

Suppose now that $\omega_\tau$ is unramified, so that
$\c=\o$ and $m>1$. Then
$\G(\omega_\tau,\psi,\varpi^{r-\v(\c)}k_m)=1$ for all $r$.
If $r>0$ then $K_0(\q\p^{-r}\cap\o)\supsetneq K_0(\q)$; since the conductor
of $\tau$ is $\q$, it follows from \cite[Theorem 5.1]{JPSS} that
\[
\int\limits_{K_0(\q\p^{-r}\cap\o)}W_{\varphi^0}(ak)\,dk=0,
\]
which in turn implies that $\Psi_r=0$.
Hence only the $r=0$ term contributes, and again we arrive at
\eqref{Psi}.

It remains only to identify the integral over $\A_m(F)$.
\begin{lemma}
When $\tau$ is ramified, we have
\[
\int_{\A_m(F)}W_{\xi^0}\!\begin{pmatrix}a&\\&I_{n-m}\end{pmatrix}
W_{\varphi^0}(a)\delta_{\B_m}(a)^{-1}\|\det{a}\|^{s-\frac{n-m}{2}}\,da=L(s,\pi\times\tau).
\]
\end{lemma}
\begin{proof}
Let $\alpha=(\alpha_1,\ldots,\alpha_n)$ and
$(\gamma_1,\ldots,\gamma_m)$ denote the Langlands
parameters of $\pi$ and $\tau$, respectively.
Since $\tau$ is ramified, we may take $\gamma_m=0$.
We set $\gamma_{m+1}=\ldots=\gamma_n=0$ and write
$\gamma=(\gamma_1,\ldots,\gamma_n)$.

Writing $a=t(a_1,\ldots,a_{m-1})$ as in \eqref{a},
by \eqref{support} we see that $W_{\varphi^0}(a)$ vanishes unless each
$a_i$ is integral. Setting
$$
\lambda_i=\sum_{1\le j\le m-i}\v(a_j)
\quad\text{for }i=1,\ldots,n,
$$
the integral in question may be written as
\[
\sum _{\substack{\lambda=(\lambda_1,\ldots,\lambda_n)\in\Z_{\ge0}^n\\
\lambda_1\ge\cdots\ge\lambda_{m-1}\\
\lambda_m=\ldots=\lambda_n=0}}
W_{\xi^0}\bigl(d_n(\lambda)\bigr)
W_{\varphi^0}\bigl(d_m(\lambda)\bigr)
\delta_{\B_m}\bigl(d_m(\lambda)\bigr)^{-1}
\|\det d_m(\lambda)\|^{s-\frac{n-m}{2}},
\]
where
$$
d_n(\lambda)=\diag\bigl(\varpi^{\lambda_1},\ldots,\varpi^{\lambda_n}\bigr)
\quad\text{and}\quad
d_m(\lambda)=\diag\bigl(\varpi^{\lambda_1},\ldots,\varpi^{\lambda_m}\bigr).
$$
On the other hand, by \cite[Theorem 4.1]{Miy}, we have
$$
W_{\xi^0}(d_n(\lambda))=\delta_{\B_n}(d_n(\lambda))^{\frac12}s_\lambda(\alpha)
\quad\text{and}\quad
W_{\varphi^0}(d_m(\lambda))=\delta_{\B_m}(d_m(\lambda))^{\frac12}
s_\lambda(\gamma),
$$
where $s_\lambda$ denotes the Schur polynomial
\[
s_\lambda(X_1,\ldots,X_n)=
\frac{\det\bigl(X_j^{\lambda_i+n-i}\bigr)_{1\le i,j\le n}}
{\prod_{1\le i<j\le n}(X_i-X_j)}
\]
(see \cite[Theorem~38.1]{bump}). Thus the integral becomes
\begin{align*}
&\sum _{\substack{\lambda=(\lambda_1,\ldots,\lambda_n)\in\Z_{\ge0}^n\\
\lambda_1\ge\cdots\ge\lambda_{m-1}\\
\lambda_m=\ldots=\lambda_n=0}}
s_\lambda(\alpha)s_\lambda(\gamma)
\left(\frac{\delta_{\B_n}(d_n(\lambda))}{\delta_{\B_m}(d_m(\lambda))}\right)^{\frac12}
\|\det d_m(\lambda)\|^{s-\frac{n-m}{2}}\\
&=\sum _{\substack{\lambda=(\lambda_1,\ldots,\lambda_n)\in\Z_{\ge0}^n\\
\lambda_1\ge\cdots\ge\lambda_n}}
s_\lambda(\alpha)s_\lambda(\gamma)
\|\det d_m(\lambda)\|^s,
\end{align*}
where the last line follows from \eqref{delta}
and the fact that $\gamma_m=\ldots=\gamma_n=0$.

By the Cauchy identity \cite[Theorem~43.3]{bump},
for sufficiently small $x\in\C$, we have
\[
\prod_{i=1}^n\prod_{j=1}^n(1-x\alpha_i\gamma_j)^{-1}
=\sum_{\substack{\lambda=(\lambda_1,\ldots,\lambda_n)\in\Z_{\ge0}^n\\
\lambda_1\ge\cdots\ge\lambda_n}}s_{\lambda}(x\alpha)s_{\lambda}(\gamma)
=\sum_{\substack{\lambda=(\lambda_1,\ldots,\lambda_n)\in\Z_{\ge0}^n\\
\lambda_1\ge\cdots\ge\lambda_n}}s_{\lambda}(\alpha)s_{\lambda}(\gamma)
x^{\lambda_1+\cdots+\lambda_n}.
\]
Therefore,
\begin{align*}
\sum _{\substack{\lambda=(\lambda_1,\ldots,\lambda_n)\in\Z_{\ge0}^n\\
\lambda_1\ge\cdots\ge\lambda_n}}
s_\lambda(\alpha)s_\lambda(\gamma)
\|\det d_m(\lambda)\|^s
=\prod_{i=1}^n\prod_{j=1}^n(1-\alpha_i\gamma_jq^{-s})^{-1}
=L(s,\pi\times\tau).
\end{align*}
\end{proof}

Finally, we take $W=W_{\overline{\xi}}$ and $W'=c^{-1}W_{\varphi^0}$
to conclude the proof of Theorem~\ref{main}.

\bibliographystyle{amsplain}
\bibliography{gln-local}

\providecommand{\bysame}{\leavevmode\hbox to3em{\hrulefill}\thinspace}
\providecommand{\MR}{\relax\ifhmode\unskip\space\fi MR }
\providecommand{\MRhref}[2]{%
  \href{http://www.ams.org/mathscinet-getitem?mr=#1}{#2}
}
\providecommand{\href}[2]{#2}
\begin{thebibliography}{10}

\bibitem{bump}
Daniel Bump, \emph{Lie groups}, second ed., Graduate Texts in Mathematics, vol.
  225, Springer, New York, 2013. \MR{3136522}

\bibitem{CPS1}
J.~W. Cogdell and I.~I. Piatetski-Shapiro, \emph{Converse theorems for {${\rm
  GL}_n$}}, Inst. Hautes \'{E}tudes Sci. Publ. Math. (1994), no.~79, 157--214.
  \MR{1307299}

\bibitem{G-K}
I.~M. Gelfand and D.~A. Kajdan, \emph{Representations of the group {${\rm
  GL}(n,K)$} where {$K$} is a local field},  (1975), 95--118. \MR{0404534}

\bibitem{JPSS}
H.~Jacquet, I.~I. Piatetski-Shapiro, and J.~Shalika, \emph{Conducteur des
  repr\'{e}sentations du groupe lin\'{e}aire}, Math. Ann. \textbf{256} (1981),
  no.~2, 199--214. \MR{620708}

\bibitem{JPSS3}
H.~Jacquet, I.~I. Piatetskii-Shapiro, and J.~A. Shalika, \emph{Rankin-{S}elberg
  convolutions}, Amer. J. Math. \textbf{105} (1983), no.~2, 367--464.
  \MR{701565}

\bibitem{J1}
Herv\'{e} Jacquet, \emph{A correction to {\it {c}onducteur des
  repr\'{e}sentations du groupe lin\'{e}aire} [mr620708]}, Pacific J. Math.
  \textbf{260} (2012), no.~2, 515--525. \MR{3001803}

\bibitem{Jan}
Fabian Januszewski, \emph{Modular symbols for reductive groups and {$p$}-adic
  {R}ankin-{S}elberg convolutions over number fields}, J. Reine Angew. Math.
  \textbf{653} (2011), 1--45. \MR{2794624}

\bibitem{KMS}
D.~Kazhdan, B.~Mazur, and C.-G. Schmidt, \emph{Relative modular symbols and
  {R}ankin-{S}elberg convolutions}, J. Reine Angew. Math. \textbf{519} (2000),
  97--141. \MR{1739728}

\bibitem{Mat2}
R.~Kurinczuk and N.~Matringe, \emph{Extension of {W}hittaker functions and test
  vectors}, Res. Number Theory \textbf{4} (2018), no.~3, Art. 31, 18.
  \MR{3831387}

\bibitem{Mat}
Nadir Matringe, \emph{Essential {W}hittaker functions for {$GL(n)$}}, Doc.
  Math. \textbf{18} (2013), 1191--1214. \MR{3138844}

\bibitem{Miy}
Michitaka Miyauchi, \emph{Whittaker functions associated to newforms for
  {$GL(n)$} over {$p$}-adic fields}, J. Math. Soc. Japan \textbf{66} (2014),
  no.~1, 17--24. \MR{3161390}

\end{thebibliography}
\end{document}